\documentclass[11pt]{article}

\usepackage{amsthm}
\usepackage{times}
\usepackage{amsmath}
\usepackage{amssymb}
\usepackage{xypic}
\usepackage{hyperref}
\usepackage[nottoc,numbib]{tocbibind}

\newtheorem{thm}{Theorem}[section]
\newtheorem{rmk}[thm]{Remark}
\newtheorem{emp}[thm]{Example}
\newtheorem{lem}[thm]{Lemma}

\newtheorem{defn}[thm]{Definition}
\newtheorem{prop}[thm]{Proposition}
\newtheorem{cor}[thm]{Corollary}

\newtheorem{step}{Step}
\begin{document}
\title{A note on cohomological boundedness for $F$-divided sheaves and $\mathcal{D}$-modules}
\author{Xiaodong Yi\footnote{yixd97@outlook.com}}
\date{}
\maketitle
\begin{abstract}
Let $X$ be a smooth proper scheme over an algebraically closed field $k$ in characteristic $p$. In this short note, by interpreting $\mathcal{D}_{X}$-modules as $F$-divided sheaves and establishing a cohomological boundedness property for $F$-divided sheaves, we prove that any $\mathcal{O}_{X}$-coherent $\mathcal{D}_{X}$-module has finite dimensional
$\mathcal{D}_{X}$-module cohomology.  
\end{abstract}
\noindent \textbf{Mathematics Subject Classification (MSC 2020):} 14F10, 14F20.\par

\tableofcontents

\section{\normalsize Introduction}

Let $X$ be a smooth scheme over an algebraically closed field $k$. The sheaf $\mathcal{D}_{X}$ of differential operators is introduced in \cite{PMIHES_1967__32__5_0}. In characteristic $0$, the sheaf $\mathcal{D}_{X}$ is generated as a sheaf of rings over $\mathcal{O}_{X}$ by operators of order $\leq 1$ and the structure as a $\mathcal{D}_{X}$-module is equivalent to the structure as a flat connection.  In characteristic $p$ this picture is not true any more. There are   (infinitely many) operators of higher order from $\mathcal{D}_{X}$ whose actions can not be recovered from the connection map. Nevertheless, such complexity essures some good properties of $\mathcal{D}_{X}$-modules instead of flat connections. For example, an $\mathcal{O}_{X}$-coherent $\mathcal{D}_{X}$-module is automatically locally free as an $\mathcal{O}_{X}$-module.  \par
This article is devoted to the study of $\mathcal{D}_{X}$-module cohomology which we define as follows. By a $\mathcal{D}_{X}$-module we mean a quasi-coherent sheaf on $X$ together with an action of $\mathcal{D}_{X}$ from the left. The category of $D_{X}$-module is denoted by $\mathcal{D}_{X}$-$\mathrm{mod}$. The structural sheaf $\mathcal{O}_{X}$ can be endowed  with a $\mathcal{D}_{X}$-module structure naturally. The category of $\mathcal{D}_{X}$-module has enough injectives and it makes sense to define  the $i$-th $\mathcal{D}_{X}$-module cohomology $H_{\mathcal{D}_{X}}^{i}$ as the $i$-th right derived functor of the following left exact functor:
\[H_{\mathcal{D}_{X}}^{0}:\mathcal{D}_{X}-\mathrm{mod} \rightarrow \mathrm{Vect}_{k},\mathcal{E}\mapsto \mathrm{Hom}_{\mathcal{D}_{X}}(\mathcal{O}_{X},\mathcal{E}),\]
where $\mathrm{Vect}_{k}$ is the category of all $k$-vector spaces, not necessarily of finite dimension. \par
Though we will not use infinitesimal crystals in the sequel, the $\mathcal{D}_{X}$-module cohomology is firstly introduced by Grothendieck in \cite{grothendieck1968crystals} from the point of view of infinitesimal crystals. In particular he proves a comparison theorem with de Rham cohomology in characteristic $0$. In characteristic $p$ this cohomology theory is not so well-behaved in the sense that it is not a Weil cohomology. To understand what the $\mathcal{D}$-module cohomology really is,  Ogus studies it in greater details in \cite{ASENS_1975_4_8_3_295_0}. In particular the cohomology of $\mathcal{O}_{X}$ has been studied in comparison with the crystalline cohomology with the same coefficient. \par
 The main result is to prove the finiteness of such cohomology assuming further $X$ is proper over $k$ and $\mathcal{E}$ is $\mathcal{O}_{X}$-coherent: 
\begin{thm}[Theorem \ref{d_mod}]
\label{intro_d_mod}
Let $X$ be a smooth proper scheme over an algebraically closed field of characteristic $p$. For any $\mathcal{O}_{X}$-coherent $\mathcal{D}_{X}$-module $\mathcal{E}$ and any cohomology degree $i$, the $\mathcal{D}_{X}$-module cohomology $H^{i}_{\mathcal{D}_{X}}(X,\mathcal{E})$ is finite dimensional as a $k$-vector space. 
\end{thm}

 Note the  analogy in characteristic $0$ is of course  trivial: under the properness assumption of $X$ over $k$ one could use the de Rham complex to compute $\mathcal{D}_{X}$-module cohomology and the finiteness follows readily from the Hodge to de Rham spectral sequence and the finiteness of coherent cohomology.  In characteristic $p$ the de Rham complex or equivalently the underlying flat connection of $\mathcal{E}$ is too weak since as already mentioned above it drops all information concerning actions on $\mathcal{E}$ of differential operators of higher order from $\mathcal{D}_{X}$ and it is not likely to give the $\mathcal{D}$-module cohomology. \par

We will proceed in a  different way. Let $F$ be the absolute Frobenius morphism on $X$. It is  observed in \cite{katz1970nilpotent} \cite{ASNSP_1975_4_2_1_1_0} that an $\mathcal{O}_{X}$-coherent $\mathcal{D}_{X}$-module can be interpreted as an infinite sequence of coherent sheaves $(\mathcal{E}_{n})_{n\geq 0}$ indexed by non-negative integers with an isomorphism $F^{*}\mathcal{E}_{n+1}\cong\mathcal{E}_{n}$ for each $n$, called a $F$-divided sheaf. We will reduce the finiteness of $\mathcal{D}_{X}$-module cohomology to the following theorem using some results from \cite{ASENS_1975_4_8_3_295_0}.  \par
For a proper scheme $X$ over $k$, a coherent sheaf $\mathcal{E}$ on $X$, and a cohomology degree $i$,  we use $h^{i}(X,\mathcal{E})$ to denote the dimension of $H^{i}(X,\mathcal{E})$ as a $k$-vector space. 
\begin{thm}[Theorem \ref{coherent}]
\label{intro_coherent}
Let $X$ be a  scheme  which is proper over an algebraically closed  field $k$ of characteristic $p$.  Let $(\mathcal{E}_{n})_{n\geq0}$ be a $F$-divided sheaf on $X$ and $\mathcal{F}$ be a coherent sheaf. Then we have
\[\sup_{n}h^{i}(X,\mathcal{E}_{n}\otimes\mathcal{F})<\infty\]
for any $i$.
\end{thm} 
The presence of an auxiliary $\mathcal{F}$ in the statement of Theorem \ref{intro_coherent} is for technical purpose. For our application to $\mathcal{D}_{X}$-modules  it suffices to take $\mathcal{F}=\mathcal{O}_{X}$. 

\subsection*{\normalsize Conventions and Notations}
\begin{description}
\item[Base field:] we will fix an algebraically closed field $k$ of characteristic $p$.  All schemes in this article are of finite type over $k$. 
\item[Frobenius:] we use $F$ to denote the absolute Frobenius morphism on an arbitrary scheme $X$ over $k$.
\item[Vector bundle:] we use interchangeably the terminologies vector bundles and locally free sheaves of finite rank. 
\end{description}
\subsection*{\normalsize Acknowledgements}
We thank Professor João Pedro dos Santos for posing the question on finiteness of infinitesimal cohomology. We thank Professor João Pedro dos Santos and Professor Niels Borne for reading the manuscript and for many discussions. 
\section{\normalsize Preliminaries}
\subsection{\normalsize $\mathcal{D}$-modules and $F$-divided sheaves}
We summarize some general facts concerning $F$-divided sheaves and $\mathcal{D}$-modules in positive characteristic, and the equivalence between them for smooth schemes. We refer to \cite{ASNSP_1975_4_2_1_1_0} \cite{AIF_2013__63_6_2267_0} for more details. 
\begin{defn}
Let $X$ be a scheme over $k$. The category $\mathrm{Fdiv}(X)$ of $F$-divided sheaves on $X$ is the category whose objects are sequences $(\mathcal{E}_{n},\sigma_{n})_{n\geq 0}$ where $\mathcal{E}_{n}$ are coherent sheaves with isomorphisms $\sigma_{n}: F^{*}\mathcal{E}_{n+1}\cong \mathcal{E}_{n}$. Arrows between two objects $(\mathcal{E}_{n},\sigma_{n})_{n\geq 0}$ and $(\mathcal{F}_{n},\tau_{n})_{n\geq 0}$ are projective systems $(\alpha_{n})_{n\geq 0}: \mathcal{E}_{n}\rightarrow \mathcal{F}_{n}$ of morphisms between $\mathcal{O}_{X}$-modules, verifying $\tau_{n}\circ F^{*}(\alpha_{n+1})=\alpha_{n}\circ\sigma_{n}$.
\end{defn}
For a $F$-divided sheaf $(\mathcal{E}_{n},\sigma_{n})_{n\geq 0}$, the isomorphisms $\sigma_{n}$ will not play a significant role in the sequel, and we will omit them from notions by simply writing $(\mathcal{E}_{n})_{n\geq 0}$ for a $F$-divided sheaf. \par
\begin{lem}[Lemma 4.2 \cite{10.2748/tmj.20200727}]
Let $(\mathcal{E}_{n})_{n\geq 0}$ be a $F$-divided sheaf. Then $\mathcal{E}_{n}$ are all vector bundles.
\end{lem}
\begin{proof}
We prove for $\mathcal{E}_{0}$. It suffices to prove that for any point $x\in X$, $\mathcal{E}_{0,x}=\mathcal{E}_{0}\otimes_{\mathcal{O}_{X}}\mathcal{O}_{X,x}$ is free as an $\mathcal{O}_{X,x}$-module.  For an ideal $I$ of $\mathcal{O}_{X,x}$ and a $p$-power $p^{n}$, we write $I^{[p^{n}]}$ to be the ideal generated by all $r^{p^{n}}$ with $r\in I$. By
\cite[\href{https://stacks.math.columbia.edu/tag/07ZD}{Tag 07ZD}]{stacks-project} it suffices to  show that, for any $i$, if the $i$-th Fitting ideal $Fit_{i}(\mathcal{E}_{x})$ of $\mathcal{E}_{x}$ is not $\mathcal{O}_{X,x}$, it is zero. Since  $(F^{n})^{*}\mathcal{E}_{n}\cong \mathcal{E}_{0}$, we have $Fit_{i}(\mathcal{E}_{0,x})=Fit_{i}(\mathcal{E}_{n,x})^{[p^{n}]}$ and then $Fit_{i}(\mathcal{E}_{0,x})\subset \mathfrak{m}_{X,x}^{p^{n}}$ for all $n\geq 0$. We conclude by noting that $\cap_{n\geq 0}\,\mathfrak{m}_{X,x}^{n}=0$.

\end{proof}
\begin{rmk} 
The lemma suggests that $F$-divided sheaves have nice categorical properties.  These properties can be pursued, for example, in the direction of Tannakian categories, for which we will not go into details in this article.  We refer to \cite{ASNSP_1975_4_2_1_1_0}\cite{dos2007fundamental}\cite{esnault2010simply}\cite{AIF_2013__63_6_2267_0}\cite{dos2015homotopy}\cite{esnault2016simply} for this topic. 
\end{rmk}
\begin{defn}
Let $X$ be smooth over $k$. We inductively define an increasing sequence $\mathcal{D}_{X}^{\leq n}$ of $k$-linear differential operators of order at most $n$ on $X$: let $\mathcal{D}_{X}^{\leq 0}=\mathcal{O}_{X}$, acting on $\mathcal{O}_{X}$ simply by multiplication. Once $\mathcal{D}_{X}^{\leq n}$ is defined, we define $\mathcal{D}_{X}^{\leq n+1}$  to be the subsheaf  of $\mathcal{E}nd_{k}(\mathcal{O}_{X})$ locally consisting of operators $D$, such that the operator $D_{a}$
\[D_{a}(x)=D(ax)-aD(x)\]
is in $\mathcal{D}_{X}^{\leq n}$ for any  local section $a\in \mathcal{O}_{X}$. 
\end{defn}
\begin{emp}
Suppose $X$ is smooth of dimension $1$, and let $x$ be an étale local coordinate of $X$. Then $\mathcal{D}_{X}$ is generated over $k[x]$ by operators 
\[D_{k}, k\geq 0, D_{k}(x^{l})=\binom{l}{k}x^{l-k}, \]
verifying the relation \[D _{k}D_{l}=\binom{k+l}{k}D_{k+l}. \]
In particular, $\mathcal{D}_{X}$ is not generated by its subsheaf $\mathcal{D}_{X}^{\leq 1}$ as an algebra over $\mathcal{O}_{X}$, contrary to the case in characteristic $0$.
\end{emp}
\begin{defn} A $\mathcal{D}_{X}$-module on $X$ is a quasi-coherent $\mathcal{O}_{X}$-module $\mathcal{E}$ on $X$, together with a morphism of algebras:
\[\mathcal{D}_{X}\rightarrow \mathcal{E}nd_{k}(\mathcal{E}). \]
\end{defn}
\begin{rmk}
In characteristic $0$, the ring of differential operators $\mathcal{D}_{X}$ is noetherian. Hence any $\mathcal{O}_{X}$-coherent $\mathcal{D}_{X}$-module is automatically coherent as a $\mathcal{D}_{X}$-module, in the sense of  being locally finitely presented.  However, this is not true in positive characteristic. It is already wrong for $\mathcal{O}_{X}$.  
\end{rmk}
The category of $\mathcal{D}_{X}$-modules, denoted by $\mathcal{D}_{X}$-$\mathrm{mod}$, is an abelian category and it has enough injective objects.
\begin{defn}
The $i$-th $\mathcal{D}_{X}$-module cohomology, denoted by $H_{\mathcal{D}_{X}}^{i}(X,\cdot)$, is the $i$-th right derived functor 
of the following left exact functor :
\[H_{\mathcal{D}_{X}}^{0}:\mathcal{D}_{X}-\mathrm{mod} \rightarrow \mathrm{Vect}_{k},\mathcal{E}\mapsto \mathrm{Hom}_{\mathcal{D}_{X}}(\mathcal{O}_{X},\mathcal{E}).\]
\end{defn}
\begin{thm}[Theorem 1.3  \cite{ASNSP_1975_4_2_1_1_0} or Theorem 1.2  \cite{AIF_2013__63_6_2267_0}]
\label{df}
Let $X$ be smooth over $k$. There is an equivalence between the full subcategory of $\mathcal{D}_{X}$-modules consisting of $\mathcal{O}_{X}$-coherent objects,  and the category $\mathrm{Fdiv}(X)$.  To be explicit, the functor going from the category of $\mathcal{D}_{X}$-modules to $Fdiv(X)$ is defined as follows: starting from a $\mathcal{D}_{X}$-module $\mathcal{E}$ we first set $\mathcal{E}_{0}=\mathcal{E}$.  For $n>0$ consider the subsheaf $\mathcal{E}_{n}$ of $\mathcal{E}$ consisting of local sections  annihilated by differential operators $D$ of order $<p^{n}$ with $D(1)=0$. The subsheaf $\mathcal{E}_{n}$ is equipped with a coherent $\mathcal{O}_{X}$-module structure by $a\cdot s=a^{p^{n}}s$, for $a\in\mathcal{O}_{X}$ and $s\in \mathcal{E}_{n}$, where the right hand side is the multiplication of $a^{p^{n}}$ on $s$, by considering $s$ as a local section of $\mathcal{E}$.  We have natural isomorphisms $F^{*}\mathcal{E}_{n+1}\cong\mathcal{E}_{n}$.
\end{thm}
\begin{lem}
\label{decr}
Let $X$ be an integral normal scheme which is proper over $k$. Let $(\mathcal{E}_{n})_{n\geq 0}$ be a $F$-divided sheaf on $X$. We have $h^{0}(X,\mathcal{E}_{n})$, $n\geq 0$, to form a decreasing (not necessarily strict) sequence. 
\end{lem}
\begin{proof}
We first deal with the case that $X$ has dimension $>1$. Use $X^{sm}$ to denote the smooth locus of $X$. By Hartogs' extension theorem we have $H^{0}(X,\mathcal{E}_{n})\cong H^{0}(X^{sm},\mathcal{E}_{n})$ and in particular the latter is finite dimensional. The natural morphism 
\[\mathcal{E}_{n}\rightarrow  F_{*}F^{*}\mathcal{E}_{n}\cong F_{*}\mathcal{E}_{n-1}\] is injective over $X^{sm}$, as this can be checked étale locally. Taking global sections we get an injection of cohomology groups 
\[H^{0}(X^{sm},\mathcal{E}_{n})\hookrightarrow H^{0}(X^{sm}, \mathcal{E}_{n-1}).\] This  is not a morphism of $k$-vector spaces but we can $k$-linearize it by making a Frobenius twist
\[H^{0}(X^{sm},\mathcal{E}_{n})\otimes_{k,F}k\hookrightarrow H^{0}(X^{sm}, \mathcal{E}_{n-1}),\]
from which we conclude. \par
It remains to prove for $X$ being a smooth proper curve. This can be proved in the same manner as above. 
\end{proof}
\subsection{\normalsize Background in intersection theory}
Following \cite{fulton2012intersection} faithfully we collect some results in intersection theory for later use. In this subsection let $X$ be a scheme which is locally of finite type and seperated over $k$. We use $A_{m}(X)_{\mathbb{Q}}$ to denote the chow group, defined as the group generated by $m$-cycles with rational coefficients, modulo rational equivalence.  We use $A_{*}(X)_{\mathbb{Q}}$ to denote the direct sum $\oplus_{m}A_{*}(X)_{\mathbb{Q}}$. Let us say 
$A_{*}(X)_{\mathbb{Q}}$ is obviously graded where the degree-$m$ piece is exactly $A_{m}(X)_{\mathbb{Q}}$. If $X$ is proper over $k$, we have a morphism 
\[\int_{X}: A_{*}(X)_{\mathbb{Q}} \rightarrow \mathbb{Q} ,\]
by first projecting to $A_{0}(X)$ and then taking the degree. \par
For us chern classes and characters are defined as morphisms, which are well-defined in the greatest generality. Following Chapter $3$  \cite{fulton2012intersection}, for any locally free sheaf $\mathcal{E}$ over $X$ of finite rank, we have the $i$-th chern class as a morphism \[ c_{i}(\mathcal{E})\cap: A_{*}(X)_{\mathbb{Q}}\rightarrow A_{*-i}(X)_{\mathbb{Q}}.\] The morphism $c_{i}(\mathcal{E})$ sends $A_{m}(X)_{\mathbb{Q}}$ to $A_{m-i}(X)_{\mathbb{Q}}$ and  it does not respect  the grading on $A_{*}(X)_{\mathbb{Q}}$. The (total) Chern character $ch(\mathcal{E})$ can be defined in terms of $c_{i}$ (Example 3.2.3 \cite{fulton2012intersection}) as well, still as a morphism from $A_{*}(X)_{\mathbb{Q}}$ to itself. \par
We shall need a Riemann-Roch theorem for singular schemes, from Chapter 18 \cite{fulton2012intersection} or the oringinal paper \cite{PMIHES_1975__45__101_0}.
\begin{thm}[Theorem 18.3, Corollary 18.3.1  \cite{fulton2012intersection}]
\label{RR}
There exists a Todd class, denoted by $td(X)$, defined as $\tau(\mathcal{O}_{X})$ where $\tau$ is the morphism from the Grothendieck group of coherent sheaves $K_{\circ}(X)$ to $A_{*}(X)_{\mathbb{Q}}$ from Theorem 18.3 \cite{fulton2012intersection}. If $X$ is proper over $k$, for any vector bundle $\mathcal{E}$ we have
\[\chi(X,\mathcal{E})=\int_{X}ch(\mathcal{E})\cap td(X). \]
\end{thm}
We make the definition of numerically trvial vector bundles.
\begin{defn}
\label{num_tri}
Let $X$ be proper over $k$. A vector bundle $\mathcal{E}$ on $X$ is said to be numerically trivial, if for all $i>0$, the homomorphism 
\[\int_{X}c_{i}(\mathcal{E})\cap: A_{*}(X)_{\mathbb{Q}}\rightarrow \mathbb{Q} \]
is zero. 
\end{defn}
\subsection{Boundedness of sheaves}
In this subsection let $X$ be an integral normal scheme, which is projective over $k$ with a fixed very ample line bundle $\mathcal{L}$. For a coherent shesf $\mathcal{E}$ we use $\mathcal{E}(t)$ to denote the tensor product $\mathcal{E}\otimes\mathcal{L}^{t}$. \par
\begin{defn}
A family of coherent sheaves $\mathcal{E}_{\iota}$, $\iota\in I$ is said to be bounded if there exists a scheme $S$ of finite type over $k$, together with a coherent sheaf $\mathcal{E}$ on $X\times S$, such that each $\mathcal{E}_{\iota}$ is isomorphic to $\mathcal{E}|_{s\times X}$ for some closed point $s\in S$. 
\end{defn}
Clearly, whenever there is a bounded family $\mathcal{E}_{\iota}$, $\iota\in I$, we obtain a  new bounded family $\mathcal{E}_{\iota}\otimes\mathcal{F}$ by tensoring with an arbitrary coherent sheaf $\mathcal{F}$. 

\begin{lem}
\label{uni_bound}
Let  $\mathcal{E}_{\iota},\iota\in I$ be a bounded family of coherent sheaves. For each degree $i$, the dimension $h^{i}(X,\mathcal{E}_{\iota})$ is uniformly bounded for all $\iota\in I$.
\end{lem}
\begin{proof}
Let $\mathcal{E}$ be a coherent sheaf on $X\times S$ as promised by definition. \par
By generic flatness we find an ascending sequence of closed subschemes of $S$:
\[\emptyset=S_{-1}\subset S_{0}\subset...\subset S_{n}=S\]
with $\mathcal{E}|_{S_{i}\setminus S_{i-1}}$ being flat over $S_{i}\setminus S_{i-1}$, where $S_{i}\setminus S_{i-1}$ is endowed with the reduced scheme structure. We immediately reduce to the case that $\mathcal{E}$ is flat over $S$, which we assume from now on. Now the uniform boundness of $h^{i}$ follows from the upper-semicontinuity of $h^{i}$, and a similar noetherian induction argument.
\end{proof}

The following is known as Kleiman criterion:
\begin{thm}[Théorème 1.13 \cite{kleiman1971theoremes}, Theorem 1.7.8 \cite{huybrechts2010geometry}]
\label{bound}
Let $\mathcal{E}_{\iota},\iota\in I$ be a family of coherent sheaves on $X$. The following are equivalent:
\begin{enumerate}
\item The family of sheaves is bounded. \par
\item Only finitely many  polynomials occur as the Hilbert polynomial $\chi(X,\mathcal{E}_{\iota}(t))$ for some $\mathcal{E}_{\iota}, \iota\in I$.  For each of these polynomial, say $P$, of degree $d$, there are constants $c_{0},...,c_{d}$  such that for each $\mathcal{E}_{\iota}$ with Hilbert polynomial $P$, there is an $\mathcal{E}_{\iota}$-regular sequence of hyperplane sections $H_{1},...,H_{d}$, with  $h^{0}(\mathcal{E}_{\iota}|_{\cap_{ j\leq i}H_{j}})\leq c_{i}$.
\end{enumerate}
\end{thm}

\section{\normalsize Boundedness of sheaves and cohomology}
\subsection{\normalsize The projective case}
The goal is to give a simple proof of the following:
\begin{prop}\label{case_normal}
Let $X$ be a normal and integral scheme which is projective over $k$. For any $F$-divided sheaf $(\mathcal{E}_{n})_{n\geq 0}$, the family of sheaves $\mathcal{E}_{n}$ parametrized by $n\geq 0$ is bounded. 
\end{prop}
Note Proposition \ref{case_normal} is by no means new.  A stronger version appears as Theorem 2.1 \cite{esnault2016simply} with an alternative proof given as Corollary 6.9 \cite{langer2025intersection}.  Both proofs are elaborated. Our proof is based on  the arguments in 
Lemma 2.2 \cite{brenner2008deep}, Lemma 2.1, Corollary 2.2 \cite{esnault2010simply}. The only modification is that, since we work with possibly singular schemes, the usual Hirzebruch-Riemann-Roch theorem should be replaced by the singular Riemann-Roch theorem.
\begin{lem}
\label{f_div_tri}
Let $(\mathcal{E}_{n})_{n\geq 0}$ be a $F$-divided sheaf on $X$. All vector bundles $\mathcal{E}_{n}$ are numerically trivial in the sense of Definition \ref{num_tri}
\end{lem}
\begin{proof}
We have to prove 
\[\int_{X}c_{i}(\mathcal{E}_{n})\cap\alpha =0\] for an arbitrary $i$-cycle class $\alpha\in A_{i}(X)_{\mathbb{Q}}$. 
Note $\int_{X}c_{i}(\mathcal{E}_{n})\cap\alpha$ have bounded denominators for $n\geq 0$. Indeed we choose a large $d$, such that $d\cdot\alpha$ is represented by a cycle with integral coefficients. Then $\int_{X}c_{i}(\mathcal{E}_{n})\cap \alpha \in \frac{1}{d}\mathbb{Z}$ for all $n\geq 0$. \par 
Since $F$ is a finite morphism, by the projection formula (Theorem 3.2 \cite{fulton2012intersection}) we have 
\[F_{*}(c_{i}(F^{*}\mathcal{E}_{n+1})\cap\alpha)=c_{i}(\mathcal{E}_{n+1})\cap F_{*}\alpha.\]
Noting $F_{*}\alpha=p^{i}\alpha$ we have 
\[\int_{X}c_{i}(\mathcal{E}_{n})\cap\alpha=p^{i}\int_{X}c_{i}(\mathcal{E}_{n+1})\cap\alpha.\]
Now $\int_{X}c_{i}(\mathcal{E}_{n})\cap\alpha =0$ is necessarily $0$ for otherwise it is  arbitrarily $p$-divisible due to the formula \[\int_{X}c_{i}(\mathcal{E}_{n})\cap\alpha=p^{im}\int_{X}c_{i}(\mathcal{E}_{n+m})\cap\alpha\] for all $m\geq 0$.
\end{proof}
\begin{lem}
\label{hil_poly}
Let $(\mathcal{E}_{n})_{n\geq 0}$ be a $F$-divided sheaf on $X$. Then we have 
\[\chi(X,\mathcal{E}_{n}(t))=\mathrm{rk}\,\mathcal{E}_{n}\cdot \chi(X,\mathcal{O}_{X}(t)).\]
\end{lem}
\begin{proof}
This is a consequence of Theorem \ref{RR}. We have 
\[\chi(X,\mathcal{E}_{n}(t))=\int_{X}ch(\mathcal{E}_{n})\cap (1+c_{1}(\mathcal{L}^{t}))\cap td(X).\] 
Unfolding the definiton $ch(\mathcal{E})=\mathrm{rk}\,\mathcal{E}_{n}+c_{1}+\frac{1}{2}(c_{1}^{2}-c_{2})+...$ by Lemma \ref{f_div_tri}. only the leading term contributes.
\end{proof}
\begin{proof}[Proof of Proposition \ref{case_normal}]
By lemma \ref{hil_poly} all $\mathcal{E}_{n}$ have the same Hilbert polynomial. It suffices to look for constants $c_{0},...,c_{d}$ as in Theorem \ref{bound}, and in the current setting $d$ is the $\mathrm{dim}X$ as well. By Bertini we choose a sequence of hyperplane sections $H_{1},...,H_{d}$, with $\cap_{j\leq i}H_{j}$ still being normal and integral for each $i$. The $F$-divided sheaf $(\mathcal{E}_{n})_{n\geq 0}$ restricts to a $F$-divided sheaf on $\cap_{j\leq i}H_{j}$ for each $i$, and the uniform boundedness of $h^{0}$ is promised by Lemma \ref{decr}.
\end{proof}
\begin{cor}\label{cohbound_proj}
Let $X$ be an integral normal scheme  which is projective over an algebraically closed  field $k$ of characteristic $p$.  Let $(\mathcal{E}_{n})_{n\geq 0}$ be a $F$-divided sheaf on $X$ and $\mathcal{F}$ be a coherent sheaf. Then we have
\[\sup_{n}h^{i}(X,\mathcal{F}\otimes\mathcal{E}_{n})<\infty\]
for any $i$. 
\end{cor}
\begin{proof}
Combine Lemma \ref{uni_bound} and Proposition \ref{case_normal}.
\end{proof}

\subsection{\normalsize Cohomological boundedness in general}
We will use the following naive bound for a spectral sequence. It will be later applied to Leray spectral sequences of maps.
\begin{lem}
\label{stupid}
Consider a spectral sequence at the second page in the first quadrant
\[E_{2}^{p,q}\Rightarrow H^{p+q},\]
whose terms are all finite dimensional spaces over $k$, and it is bounded in the sense that $E_{2}^{s,t}\neq 0$ only if $s\leq M$ and $t\leq N$. We have the inequalities
\[\mathrm{dim}H^{n}\leq \mathrm{dim}E_{2}^{n,0}+\sum_{0\leq i\leq n-1}\mathrm{dim}E_{2}^{i,(n-i)}\]
and \[ \mathrm{dim} E_{2}^{n,0}\leq \mathrm{dim}H^{n}+\sum_{2\leq i\leq N}\mathrm{dim}E_{2}^{n-i,i-1}.\]
\end{lem}
\begin{proof}
The first inequality is trivial. For the second note
\[\mathrm{dim}E_{i+1}^{n,0}\geq \mathrm{dim}E_{i}^{n,0} -\mathrm{dim}E_{i}^{n-i,i-1}\geq \mathrm{dim}E_{i}^{n,0} -\mathrm{dim}E_{2}^{n-i,i-1}.\] The spectral sequence necessarily degenerates at $(N+1)$-th page and we have
\[\mathrm{dim}H^{n}\geq \mathrm{dim}E_{N+1}^{n,0}\geq \mathrm{dim}E_{N}^{n,0}-\mathrm{dim}E_{2}^{n-N,N-1} \geq ...\geq \mathrm{dim}E_{2}^{n,0}-\sum_{2\leq i\leq N}\mathrm{dim}E_{2}^{n-i,i-1}.\]
\end{proof}
\begin{thm}
\label{coherent}
Let $X$ be a scheme  which is proper over an algebraically closed  field $k$ of characteristic $p$.  Let $(\mathcal{E}_{n})_{n\geq 0}$ be a $F$-divided sheaf on $X$ and $\mathcal{F}$ be a coherent sheaf on $X$. Then we have
\[\sup_{n}h^{i}(X,\mathcal{F}\otimes\mathcal{E}_{n})<\infty\]
for any $i$.
\end{thm}
\begin{proof}
In the proof, we use \[\lfloor{X},(\mathcal{E}_{n})_{n\geq 0},\mathcal{F}\rfloor\]
to denote the property \[\sup_{i\geq 0,n\geq 0}h^{i}(X,\mathcal{F}\otimes\mathcal{E}_{n})<\infty \]
as in the statement of Theorem \ref{coherent}.
\begin{step}
\label{projective}
We have $\lfloor X,(\mathcal{E}_{n})_{n\geq 0},\mathcal{F}\rfloor$ whenver $X$ is projective, normal and integral. This is Corollary \ref{cohbound_proj}. 
\end{step}
We apply induction on $d=\mathrm{dim}X$. The case of $d=0$ is  trivial. 

\begin{step}
\label{reduced}
It suffices to prove $\lfloor X ,(\mathcal{E}_{n})_{n\geq 0},\mathcal{F}\rfloor$ for reduced schemes $X$. To prove this, let  $\mathcal{I}$ be the sheaf of ideals defining the reduced scheme structure $r: X^{red}\hookrightarrow X$.  We have a finite ascending sequence \[0=\mathcal{I}^{m}\mathcal{F}\subset \mathcal{I}^{m-1}\mathcal{F}\subset...\subset \mathcal{I}^{0}\mathcal{F}=\mathcal{F},\] and it is clear that it suffices to check the property $\lfloor X ,(\mathcal{E}_{n})_{n\geq 0},\mathcal{I}^{l}\mathcal{F}/\mathcal{I}^{l+1}\mathcal{F}\rfloor$ for $0\leq l\leq m-1$, which is equivalent to the property  $\lfloor X^{red} ,(r^{*}\mathcal{E}_{n})_{n\geq 0},r^{*}(\mathcal{I}^{l}\mathcal{F}/\mathcal{I}^{l+1}\mathcal{F})\rfloor$ for $0\leq l\leq m-1$.
\end{step}
\begin{step}
\label{integral}
It suffices to check $\lfloor X ,(\mathcal{E}_{n})_{n\geq 0},\mathcal{F}\rfloor$  for normal integral schemes $X$. By Step \ref{reduced} we assume $X$ is reduced and consider the normalization map $\nu: X^{\nu}\rightarrow X$.  The morphism $\nu$ is finite, since the normalization maps are finite for Nagata schemes (\cite[\href{https://stacks.math.columbia.edu/tag/035S}{Tag 035S}]{stacks-project}
) and schemes of finite type over a field are Nagata (\cite[\href{https://stacks.math.columbia.edu/tag/035B}{Tag 035B}]{stacks-project}
).  By \cite[\href{https://stacks.math.columbia.edu/tag/0BXC}{Tag 0BXC}]{stacks-project} we have an open subset $U$ of $X$ containing all generic points of $X$ over which $\nu$ is an isomorphism.  We have $\lfloor X ,(\mathcal{E}_{n})_{n\geq 0},\mathcal{F}\rfloor$  to be equivalent to $\lfloor X ,(\mathcal{E}_{n})_{n\geq 0},\nu_{*}\nu^{*}\mathcal{F}\rfloor$  since the canonical morphism $\mathcal{F}\rightarrow \nu_{*}\nu^{*}\mathcal{F}$ restricts to an isomorphism on $U$ and we apply induction to its kernel and cokernel, which have supports of dimension $<d$. The property  $\lfloor X ,(\mathcal{E}_{n})_{n\geq 0},\nu_{*}\nu^{*}\mathcal{F}\rfloor$ is exactly  $\lfloor X^{\nu} ,(\nu^{*}\mathcal{E}_{n})_{n\geq 0},\nu^{*}\mathcal{F}\rfloor$ since $\nu$ is finite. 
\end{step}
\begin{step}
We assume $X$ is integral and normal. By Chow's lemma there exists $f:X'\rightarrow X$ with $X'$ being integral, normal and projective over $k$. There exists an open subset $U$ of $X$ over which $f$ is an isomorphism.  Consider the Leray spectral sequence for $f$ and the sheaf $f^{*}\mathcal{F}\otimes f^{*}\mathcal{E}_{n}$:  
\[E_{2}^{pq}=H^{p}(X,R^{q}f_{*}f^{*}\mathcal{F}\otimes\mathcal{E}_{n}) \Rightarrow H^{p+q}(X',f^{*}\mathcal{F}\otimes f^{*}\mathcal{E}_{n}).\] 
We have $\lfloor X ,(\mathcal{E}_{n})_{n\geq 0},\mathcal{F}\rfloor$  to be equivalent to $\lfloor X ,(\mathcal{E}_{n})_{n\geq 0},f_{*}f^{*}\mathcal{F}\rfloor$  since the canonical morphism $\mathcal{F}\rightarrow f_{*}f^{*}\mathcal{F}$ restricts to an isomorphism on $U$ and we apply induction to its kernel and cokernel, which have supports of dimension $<d$. Note we have $R^{q}f_{*}=0$ over $U$ for $q>0$. Applying induction to  $\lfloor X ,(\mathcal{E}_{n})_{n\geq 0},R^{q}f_{*}f^{*}\mathcal{F}\rfloor$ for $q>0$ and by Lemma \ref{stupid} we see  that $\lfloor X ,(\mathcal{E}_{n})_{n\geq 0},f_{*}f^{*}\mathcal{F}\rfloor$ is equivalent to $\lfloor X' ,(f^{*}\mathcal{E}_{n})_{n\geq 0},f^{*}\mathcal{F}\rfloor$. Finally $\lfloor X',(\mathcal{E}_{n})_{n\geq 0},f^{*}\mathcal{F}\rfloor$ holds by 
 Step \ref{projective}. 
\end{step}
\end{proof}

\begin{cor}
\label{coherent_cor}
Let $X$ be a smooth proper scheme over $k$. For any $F$-divided sheaf $(\mathcal{E}_{n})_{n\geq 0}$ we have 
\[\sup_{n}h^{i}(X,\mathcal{E}_{n})<\infty\]
for any $i$.
\end{cor}
\section{\normalsize The finiteness of $\mathcal{D}$-module cohomology}
Still let $X$ be a scheme over $k$ and let $(\mathcal{E}_{n})_{n\geq0}$ be a $F$-divided sheaf. 
We consider the $n$-fold Frobenius twist  $ H^{i}(X,\mathcal{E}_{n})\otimes_{k,F^{n}}k$, with $\alpha x\otimes\beta=x\otimes \alpha^{p^{n}}\beta$ and  the $k$-linear structure  given by the right tensor factor.
The natural morphism bwtween cohomology
\[H^{i}(X,\mathcal{E}_{n+1})\rightarrow H^{i}(X, F^{*}\mathcal{E}_{n+1})\cong H^{i}(X,\mathcal{E}_{n}),\]
can be $k$-linearized as a morphism
\[H^{i}(X,\mathcal{E}_{n+1})\otimes_{k,F^{n+1}}k \rightarrow H^{i}(X,\mathcal{E}_{n})\otimes_{k,F^{n}}k,\]
and we obtain an inverse system $\{H^{i}(X,\mathcal{E}_{n})\otimes_{k,F^{n}}k\}_{n\geq 0}$. 
\begin{lem}[Theorem 2.4 \cite{ASENS_1975_4_8_3_295_0}]
\label{exactsequence}
Let $X$ be smooth over $k$. Let $\mathcal{E}$ be an $\mathcal{O}_{X}$-coherent $\mathcal{D}_{X}$-module and  $(\mathcal{E}_{n})_{n\geq0}$ be the corresponding $F$-divided sheaf on $X$ with $\mathcal{E}_{0}=\mathcal{E}$ (Theorem \ref{df}). For any $i$, we have an exact sequence
\[0\rightarrow R^{1}\varprojlim_{n} H^{i-1}(X,\mathcal{E}_{n})\otimes_{k,F^{n}}k \rightarrow H_{\mathcal{D}_{X}}^{i}(X,\mathcal{E}) \rightarrow \varprojlim_{n}H^{i}(X,\mathcal{E}_{n})\otimes_{k,F^{n}}k\rightarrow 0. \]
\end{lem}
\begin{rmk}
See also \cite{mundinger2024hochschild} for Lemma \ref{exactsequence} in the context of Hochschild cohomology. 
\end{rmk}
We are ready to prove the main theorem. 
\begin{thm}
\label{d_mod}
Let $X$ be a smooth proper scheme over an algebraically closed field of characteristic $p$. For any $\mathcal{O}_{X}$-coherent $\mathcal{D}_{X}$-module $\mathcal{E}$ and any cohomology degree $i$, the $\mathcal{D}_{X}$-module cohomology $H^{i}_{\mathcal{D}_{X}}(X,\mathcal{E})$ is finite dimensional as a $k$-vector space. 
\end{thm}
\begin{proof}
Still write $(\mathcal{E}_{n})_{n\geq 0}$ the $F$-divided sheaf corresponding to $\mathcal{E}$. 
 Since $H^{i-1}(X,\mathcal{E}_{n})$, $n\geq 0$, are all of finite dimension, the inverse system \[H^{i-1}(X,\mathcal{E}_{n})\otimes_{k,F^{n}}k,n\geq 0\] verifies the Mittag-Leffler condition (Definition 3.5.6  \cite{weibel1994introduction}) and the first derived limit $R^{1}\varprojlim_{n} H^{i-1}(X,\mathcal{E}_{n})\otimes_{k,F^{n}}k =0$ (Proposition 3.5.7 \cite{weibel1994introduction}) .  \par
The inverse limit $\varprojlim_{n}H^{i}(X,\mathcal{E}_{n})\otimes_{k,F^{n}}k$ is finite dimensional due to Corollary \ref{coherent_cor} and Lemma \ref{inverselimit} below. 
\end{proof}
\begin{lem} \label{inverselimit}
Suppose we have an inverse system of finite dimensional vector spaces over $k$ indexed by the category $\mathbb{N}$. Namely, we have a finite dimensional vector space $V_{i}$ for each $i\in \mathbb{N}$ and an arrow $f_{ij}:V_{i}\rightarrow V_{j}$ whenver $i>j$. Suppose we have 
\[\sup_{i}\,\mathrm{dim}_{k}V_{i}<\infty,\]
then the dimension of $\varprojlim_{i}\,V_{i}$ is bounded by $\sup_{i} \mathrm{dim}_{k}V_{i}$.
\end{lem}

\begin{proof}
For each $i$ we use $V_{i}^{s}$ to denote the intersection \[V_{i}^{s}=\cap_{j>i}\mathrm{Im}\,f_{ji},\] called the stable subspace of $V_{i}$.  Then obviously the subspaces $V_{i}^{s}$ form an inverse system indexed by $\mathbb{N}$ as well, with transition maps  all surjective, and 
\[\varprojlim_{i}V^{s}_{i}=\varprojlim_{i}V_{i}.\]
The transition map $V_{i+1}^{s}\rightarrow V_{i}^{s}$ is an isomorphism for any sufficiently large $i$, otherwise due to the surjectivity of these transition maps the dimension of $V_{i}^{s}$ will be unbounded, contradicting to the assumption. The dimension of the inverse limit is also bounded by the dimension of $V_{i}^{s}$ for sufficiently large $i$, hence bounded by $\sup_{i}\,\mathrm{dim}_{k}V_{i}$.
\end{proof}
\begin{rmk}
The $0$-th $\mathcal{D}_{X}$-module cohomology  has a relative analogy named as the Gauss-Manin stratification in \cite{AIF_2013__63_6_2267_0} where the author proves a finiteness result  even without  assuming the family to be proper (see lemma 1.5 \cite{AIF_2013__63_6_2267_0}).\par
\end{rmk}
\begin{rmk}
Theorem \ref{d_mod} is somewhat the best we can achieve. To generalize it naively in the following two natural directions both produces something pathological. Namely
\begin{enumerate}
\item What  if $X$ is smooth but not necessarily proper over $k$? In this case the $\mathcal{D}_{X}$-module cohomology can be infinite dimensional. For example, take $X=\mathbb{A}_{k}^{1}$, $i=1$ and $\mathcal{M}=\mathcal{O}_{X}$. Apply Lemma \ref{exactsequence}.
\item Suppose we have a smooth proper morphism $X\rightarrow S$. In  \cite{AIF_2013__63_6_2267_0} the author defines the $0$-th Gauss-Manin stratification as a functor from $\mathcal{O}_{X}$-coherent $\mathcal{D}_{X}$-modules to  $\mathcal{O}_{S}$-coherent $\mathcal{D}_{S}$-modules. What about higher (or derived) Gauss-Manin stratifications? This is probably problematic in general, meaning the higher Gauss-Manin stratification may not be coherent as $\mathcal{O}_{S}$-module if we start from a $\mathcal{O}_{X}$-coherent $\mathcal{D}_{X}$-module. It suffices to look at Example 3.7 \cite{ASENS_1975_4_8_3_295_0}.
\end{enumerate}
\end{rmk}
\begin{rmk}
We refer to \cite{bao2025cohomology} for some computations and a $K(\pi,1)$-problem for the  $\mathcal{D}$-module cohomology on curves. 
\end{rmk}

 \bibliographystyle{plain}
 \bibliography{D-module}

\end{document}